\documentclass[11pt,reqno,a4paper]{amsart}
\usepackage{amsmath,amsfonts,amssymb,amsthm,amsxtra,latexsym,amscd,enumerate,verbatim}
\usepackage[margin=1in]{geometry}
\usepackage[abbrev,lite,nobysame]{amsrefs}
\usepackage{graphics,graphicx}
\usepackage{subcaption}
\usepackage{tabularx}
\usepackage[usenames,dvipsnames]{color}
\usepackage{times}
\usepackage{bbm}
\usepackage[colorlinks=true, pdfstartview=FitV, linkcolor=blue, citecolor=blue, urlcolor=blue]{hyperref}
\usepackage{tikz}
\usepackage{enumerate,enumitem}
\usepackage{mathtools}
\usepackage[utf8]{inputenc}
\usepackage{amsmath,amsfonts,amssymb,amsxtra,latexsym,amscd,enumerate,amsthm,verbatim}
\usepackage[abbrev,lite,nobysame]{amsrefs}
\usepackage{float}
\usepackage{tikz}
\usepackage{pgfplots}
\usepackage{caption}
\usepackage[colorlinks=true]{hyperref}
\hypersetup{urlcolor=blue, citecolor=blue}
\usepackage{hyperref}
\numberwithin{equation}{section}
\def\dd{{\mathrm{d}}}

\newcommand{\R}{\mathbb{R}}

\newcommand{\defeq}{\mathrel{\mathop:}=}

\newcommand{\be}{\begin{equation}}
\newcommand{\ee}{\end{equation}}
\newcommand{\Per}{\operatorname{Per}}

\newtheorem{lemma}{Lemma}[section]
\newtheorem{prop}{Proposition}[section]
\newtheorem{theorem}{Theorem}
\newtheorem{corollary}{Corollary}[section]
\newtheorem{remark}{Remark}[section]

\pgfplotsset{compat=1.18} 
\title[An exact Delaunay flux identity]{An exact Delaunay flux identity for compact equilibria in the liquid drop model}

\author{Averkios Averkiou}
\address{Averkios Averkiou, Department of Mathematical Sciences, University of Bath, BA2 7AY, United Kingdom}
\email{aa4119@bath.ac.uk}

\author{Manuel del Pino}
\address{Manuel del Pino, Department of Mathematical Sciences, University of Bath, BA2 7AY, United Kingdom}
\email{mdp59@bath.ac.uk}

\author{Monica Musso}
\address{Monica Musso, Department of Mathematical Sciences, University of Bath, BA2 7AY, United Kingdom}
\email{mm2683@bath.ac.uk}
\begin{document}
\begin{abstract}
We revisit the construction of compact liquid drop equilibria from Delaunay unduloids in \cite{dPMZ}. We evaluate exactly the coefficient arising in the final solvability condition of the Lyapunov-Schmidt reduction, identifying it as the product of the axial period and the normalised Delaunay flux. Using this identity and the monotonicity of the period, we establish the conjectured positivity of the coefficient and prove that its product with the period strictly increases with the neck radius. We also derive an exact relation between the volume and the surface area of one period and prove that both strictly decrease with the neck radius. These results extend the construction throughout the embedded unduloid family and yield compact equilibria of every sufficiently large mass, with reference neck radii converging to any prescribed interior value as the mass tends to infinity.
\end{abstract}

\maketitle

\section{Introduction}
The classical liquid drop model of Gamow \cite{Gamow}, subsequently developed by Bohr and Wheeler \cite{BohrWheeler}, describes an atomic nucleus as an incompressible drop subject to the competing effects of surface tension and Coulomb self-repulsion. For a drop occupying a bounded domain $\Omega\subset\R^3$, the associated energy reads
\begin{equation}\label{liquidenergy}
\mathcal E(\Omega)=\Per(\Omega)+\frac{1}{2}\int_\Omega\int_\Omega\frac{\dd x\,\dd y}{|x-y|},
\end{equation}
where $\Per(\Omega)$ denotes the De Giorgi perimeter, which is equal to the surface area when $\Omega$ has smooth boundary. Given $m>0$, smooth critical points of  \eqref{liquidenergy} with prescribed volume $|\Omega|=m$ satisfy the Euler-Lagrange equation
\begin{equation}\label{liquidequation}
H_{\partial\Omega}(x)+\int_\Omega\frac{\dd y}{|x-y|}=\lambda,\quad x\in\partial\Omega,
\end{equation}
where $H_{\partial\Omega}$ denotes the mean curvature of $\partial\Omega$ and $\lambda$ is the Lagrange multiplier associated with the volume constraint. For a survey of the liquid drop model, we refer to \cite{CMT}.

Balls satisfy \eqref{liquidequation} for every prescribed mass $m>0$, and their minimality and stability properties have been investigated in \cites{KnupferMuratov,Julin,BonaciniCristoferi,FigalliFuscoMaggiMillotMorini,ChodoshRuohoniemi}. Further work on the associated global minimisation problem addresses compactness and existence \cites{FrankLieb,FrankNam}, as well as nonexistence for sufficiently large masses \cites{KnupferMuratov,LuOtto,FrankKillipNam,Schulz}. Notably, the recent work \cite{ChodoshGianocca} established the threshold for existence and spherical minimality conjectured in \cite{ChoksiPeletier}, proving that, up to translation, balls are the unique global minimisers for $0<m\leq m_{*}=5(2-2^{\frac{2}{3}})(2^{\frac{2}{3}}-1)^{-1}$, while no global minimiser exists for $m>m_{*}$.

The sharp characterisation of the mass threshold for global minimality naturally motivates the study of non-spherical critical points of \eqref{liquidenergy}. Rigorous examples include the toroidal and double-toroidal solutions constructed in \cites{RenWeiTorus,RenWeiDouble} and the branches bifurcating from balls in \cites{FrankNonspherical,DeRegibusGrossiMusso}, while the numerical investigations in \cite{XuDu} also exhibit toroidal and Delaunay-type configurations. Finally, gluing methods have produced two further families of non-spherical equilibria, with \cite{dPMZ} constructing solutions for all sufficiently large masses that resemble a ``pearl necklace'' of Delaunay unduloids, and \cite{DelPinoFrankMusso} obtaining axially symmetric solutions for all sufficiently small masses, resembling two equal spheres joined by an approximately catenoidal neck.

\section{Delaunay unduloids}
The present work revisits the construction in \cite{dPMZ}
which produces compact equilibria satisfying
\eqref{liquidequation} for every sufficiently large mass,
with boundaries obtained by bending small normal perturbations of long
Delaunay unduloid segments around a large circle. Throughout, we use the convention that the mean curvature $H$ equals the sum of the principal curvatures.

Let $\Sigma(a)\subset\mathbb R^3$ denote the embedded Delaunay unduloid with neck radius $a\in\left(0,\frac 12\right)$ and constant mean curvature $H\equiv2$. Choosing the axial origin at a bulge, we parametrise $\Sigma(a)$ by
\begin{equation}\label{delaunayparametrisation}
Y_a(\theta,s)=\bigl(f_a(s)\cos\theta,f_a(s)\sin\theta,s\bigr),\quad (\theta,s)\in[0,2\pi)\times\mathbb R,
\end{equation}
where the generating profile satisfies the Cauchy problem
\begin{equation}\label{profileequation}
-\frac{f_a''}{\left(1+(f_a')^2\right)^{\frac32}}+\frac{1}{f_a\sqrt{1+(f_a')^2}}=2,\quad f_a(0)=1-a,\quad f_a'(0)=0.
\end{equation} The function $f_a:\R\to[a,1-a]$ is smooth, positive, even and periodic, with axial period hereafter denoted by $T_a$. In addition, $f_a$ is strictly decreasing in $(0,\frac{T_a}{2})$, with $
f_a\Big(\frac{T_a}{2}\Big)=a$ and $
f_a'\Big(\frac{T_a}{2}\Big)=0.$

In what follows, for ease of notation we write $f=f_a$, consider the parameters
\begin{equation*}
b\defeq1-a,\qquad c\defeq ab=a(1-a),
\end{equation*}
and denote the volume and surface area  associated
with one axial period $T_a$ by
\begin{equation}\label{periodgeometry}
V_a\defeq \pi\int_{-\frac{T_a}{2}}^{\frac{T_a}{2}}f_a(s)^2\,\dd s,
\qquad A_a\defeq 2\pi\int_{-\frac{T_a}{2}}^{\frac{T_a}{2}} f_a(s)\sqrt{1+(f_a'(s))^2}\,\dd s.
\end{equation}
A direct computation shows that the Delaunay equation
\eqref{profileequation} admits the conserved quantity
\begin{equation}\label{conservedquantity}
f^2-\frac{f}{\sqrt{1+(f')^2}}=-c,
\end{equation}
while solving for $f'$ on the decreasing half-period further gives
\begin{equation}\label{decreasingbranch}
f'=-\frac{\sqrt{(f^2-a^2)(b^2-f^2)}}{f^2+c},
\qquad 0<s<\frac{T_a}{2}.
\end{equation}
Furthermore, for $\gamma>0$ the rescaling $\Omega=\gamma^{\frac13}\widetilde\Omega$ and a redefinition of the Lagrange multiplier transform \eqref{liquidequation} into the rescaled Euler-Lagrange equation\begin{equation}\label{scaledequation}
H_{\partial\widetilde\Omega}(x)+\gamma\int_{\widetilde\Omega}\frac{\dd y}{|x-y|}=\lambda,\quad x\in\partial\widetilde\Omega, \qquad|\widetilde\Omega|=\gamma^{-1}m.
\end{equation}
For fixed $a\in(0,\frac12)$ and sufficiently large $n\in\mathbb N$,
the construction in \cite{dPMZ} bends a segment containing $n$ complete periods
of $\Sigma(a)$ around a circle of radius $R=\frac{nT_a}{2\pi}$
and applies a Lyapunov-Schmidt reduction to obtain a nearby equilibrium. The final scalar solvability condition determines the coupling parameter $\gamma$ and involves the coefficient
\begin{equation}\label{coefficient}
I_a\defeq \int_0^{\frac{T_a}{2}}
\frac{f}{\left(1+(f')^2\right)^{\frac52}}\Big[ff''\bigl(2-(f')^2\bigr)
+\bigl(1+3(f')^2\bigr)\bigl(1+(f')^2\bigr)\Big]\,\dd s.
\end{equation}
Provided $I_a>0$, \cite{dPMZ}*{Theorem 2} yields a bounded domain $\widetilde\Omega_{n,a}$ with smooth embedded boundary for every sufficiently large integer $n$. Its boundary satisfies the Euler-Lagrange equation in \eqref{scaledequation} for a coupling parameter $\gamma_{n,a}>0$ and a Lagrange multiplier $\lambda_{n,a}\in\mathbb R$ satisfying
\begin{equation}\label{dPMZestimates}
|\widetilde\Omega_{n,a}|=nV_a\bigl(1+o(1)\bigr), \qquad
\gamma_{n,a} =\mathcal{O}\left(\frac{1}{\log n}\right) \qquad\text{as }n\to\infty,
\end{equation}
with $o(1)\to 0$ as $n\to \infty$. Moreover, for any compact interval $K\subset(0,\frac12)$ such that
$\inf\limits_{a\in K}I_a>0$, the estimates in \eqref{dPMZestimates} are uniform in $K$, with continuous dependence on $a\in K$.
Returning to the original scaling gives an equilibrium $\Omega_{n,a}=(\gamma_{n,a})^{\frac13}\widetilde\Omega_{n,a}$
with physical mass
\[m_n(a)\defeq|\Omega_{n,a}|=\gamma_{n,a}|\widetilde\Omega_{n,a}|.\]
\subsection*{Statement of the main result}
The positivity condition for the coefficient $I_a$ required above was verified in \cite{dPMZ}*{Lemma 2.1} for
\[a\in(0,a_*]\cup\bigg[\frac14,\frac12\bigg),\]
where $a_*\in(0,\frac14)$ is sufficiently small, while the same result also provides the small-neck asymptotics $I_a=2a+o(a)$ as $a\to0^+$. Nevertheless, the positivity of $I_a$ was not established in $(a_*,\frac14)$, leaving the applicability of the construction in this range unresolved.

In addition, prescribing an arbitrary sufficiently large physical mass $m$ requires varying the reference neck, since fixing $a$ produces only a sequence of masses $m_n(a)$ indexed by $n$. The argument in \cite{dPMZ}*{Proof of Theorem 1} combines the small-neck asymptotics with the continuous dependence on $a$ to adjust the reference neck and realise every sufficiently large mass. Extending this argument to a neighbourhood of any interior reference neck requires control of how the leading mass coefficient varies with $a$, motivating the conjectures stated in \cite{dPMZ}*{Remark 2.2}, namely
\[I_a>0,\qquad \frac{\dd}{\dd a}(T_aI_a)>0
\qquad\text{for every }a\in\Big(0,\frac12\Big).\]
Our main result gives an exact formula for $I_a$ and establishes both conjectures.
\begin{theorem}\label{maintheorem}
For every $a\in(0,\frac12)$, the coefficient $I_a$ in \eqref{coefficient} satisfies the exact identity
\begin{equation}\label{fluxidentity}
I_a=a(1-a)T_a>0,
\end{equation}
and
\[\frac{\dd T_a}{\dd a}>0,\qquad \frac{\dd I_a}{\dd a}>0, \qquad \frac{\dd}{\dd a}\bigl(2T_aI_a\bigr)>0.
\]
In particular, the maps
\[
T_a:\Big(0,\frac12\Big)\rightarrow(2,\pi),\qquad I_a:\Big(0,\frac12\Big) \rightarrow\Big(0,\frac{\pi}{4}\Big),\qquad 2T_aI_a:\Big(0,\frac12\Big)
\rightarrow\Big(0,\frac{\pi^2}{2}\Big)
\]
are strictly increasing bijections.
\end{theorem}
\begin{remark}
The identity \eqref{fluxidentity} gives an exact geometric interpretation of the coefficient $I_a$ arising in the scalar solvability condition by expressing it as the product of the axial period and the normalised Delaunay flux. Indeed, with the positive axial orientation and the convention $H\equiv2$, the axial flux across any circular section is given by
\[
\mathcal F_a\defeq 2\pi\bigg(\frac{f_a}{\sqrt{1+(f_a')^2}}-f_a^2\bigg)=2\pi a(1-a),
\]
thus the constant $c=a(1-a)$ in \eqref{conservedquantity} is the normalised axial flux $\frac{\mathcal F_a}{2\pi}$, and the identity $I_a=cT_a$ takes the form
\[
I_a=\frac{T_a}{2\pi}\mathcal F_a.
\]
\end{remark}
\begin{remark}
The factorisation $T_aI_a=a(1-a)T_a^2$ reduces the conjectured monotonicity of $T_aI_a$ to showing that $\frac{\dd T_a}{\dd a}>0$. Together with the volume monotonicity $\frac{\dd V_a}{\dd a}<0$ established in Proposition \ref{volumemonotonicity}, this allows the prescribed-mass argument of \cite{dPMZ} to be applied near any fixed $a_0\in(0,\frac12)$, yielding equilibria of every sufficiently large mass $m$ with reference necks $a_m$ satisfying $a_m\to a_0$ as $m\to\infty$.
\end{remark}

\subsection{Organisation of the paper}
Section \ref{sectionexactidentity} proves the exact flux identity \eqref{fluxidentity}, while Section \ref{monotonicity} develops the arclength representation and establishes the monotonicity properties of the axial period $T_a$, the solvability coefficient $I_a$, and the volume $V_a$ and surface area $A_a$ of one period. Section \ref{consequences} applies these results to the construction in \cite{dPMZ}. Finally, the appendix provides an independent proof of the flux identity using complete elliptic integrals.
\section{The exact flux identity}\label{sectionexactidentity}
Recalling the equation \eqref{profileequation} satisfied by the Delaunay generating profile $f=f_a$, we define
\begin{equation}\label{defphialpha}
\Phi_a(s)\defeq\frac{f}{\left(1+(f')^2\right)^{\frac{5}{2}}}
\Big[ ff''\bigl(2-(f')^2\bigr) +\bigl(1+3(f')^2\bigr)\bigl(1+(f')^2\bigr)
\Big], \quad s\in\mathbb R,
\end{equation}
so that the coefficient $I_a$ in \eqref{coefficient} reads
\[
I_a=\int_0^{\frac{T_a}{2}}\Phi_a(s)\,\dd s.
\]
The key algebraic observation is that the integrand $\Phi_a$ admits a decomposition into a multiple of the normalised Delaunay flux and an exact derivative.
\begin{lemma}\label{pointwiseidentity}
Fix $a\in\left(0,\frac{1}{2}\right)$. Then, for every $s\in\mathbb R$ it holds
\begin{equation}\label{fluxidentitypointwise}
\Phi_a(s)=2c+\frac{\dd}{\dd s}\left(\frac{f^2f'}{\left(1+(f')^2\right)^{\frac{3}{2}}}\right),
\end{equation}
and
\begin{equation}\label{exactidentity}
I_a=cT_a=a(1-a)T_a>0.
\end{equation}
\end{lemma}

\begin{proof}
We first note that the structure of \eqref{defphialpha} suggests differentiating
$\frac{f^2f'}{\left(1+(f')^2\right)^{\frac{3}{2}}}$, whose derivative contains a term involving $\frac{f''}{\left(1+(f')^2\right)^{\frac{5}{2}}}$. More precisely, a direct calculation gives
\[
\frac{\dd}{\dd s}\left(\frac{f^2f'}{\left(1+(f')^2\right)^{\frac{3}{2}}}
\right)=\frac{2f(f')^2\left(1+(f')^2\right)+
f^2f''\left(1-2(f')^2\right)}{\left(1+(f')^2\right)^{\frac{5}{2}}},
\]
which we combine with \eqref{defphialpha} to obtain
\begin{equation}\label{firstidentitylemma1}
\Phi_a-\frac{\dd}{\dd s}\left(\frac{f^2f'}{\left(1+(f')^2\right)^{\frac{3}{2}}}
\right)=\frac{f\left(ff''+1+(f')^2\right)}
{\left(1+(f')^2\right)^{\frac{3}{2}}}.
\end{equation}
On the other hand, multiplying \eqref{profileequation} by $f\left(1+(f')^2\right)^{\frac{3}{2}}$ yields $
ff''=1+(f')^2-2f\left(1+(f')^2\right)^{\frac{3}{2}}$, hence substituting into \eqref{firstidentitylemma1} one finds
\[\Phi_a-\frac{\dd}{\dd s}\left(\frac{f^2f'}{\left(1+(f')^2\right)^{\frac{3}{2}}}
\right)=2\frac{f}{\sqrt{1+(f')^2}}-2f^2.\]
The conserved quantity \eqref{conservedquantity} implies that the right-hand side of the above expression is equal to $2c=2a(1-a)$, which proves \eqref{fluxidentitypointwise}.

To conclude, we integrate \eqref{fluxidentitypointwise} over $\left[0,\frac{T_a}{2}\right]$ to obtain
\[I_a=cT_a+\left[\frac{f^2f'}{\left(1+(f')^2\right)^{\frac{3}{2}}}
\right]_{0}^{\frac{T_a}{2}}=a(1-a)T_a>0,\]
where we have used that $s=0$ and $s=\frac{T_a}{2}$ correspond to a bulge and its adjacent neck, respectively, hence
$f'(0)=f'\left(\frac{T_a}{2}\right)=0$.
\end{proof}
\section{Arclength representation and monotonicity}
\label{monotonicity}

Having established the exact identity $I_a=a(1-a)T_a $ in Lemma \ref{pointwiseidentity}, we devote this section to studying the dependence of the Delaunay period $T_a$ on the neck parameter $a$ and proving the remaining monotonicity and range assertions in Theorem \ref{maintheorem}. Throughout this section,  we retain the notation $f_a$ in order to make the dependence on the neck parameter explicit.

We recall from  \eqref{delaunayparametrisation} that the meridian generating curve of $\Sigma(a)$ is parametrised by
\[s\mapsto \bigl(f_a(s),s\bigr), \quad s\in\R,\]
where $s$ is the Euclidean coordinate along the axis of revolution and $f_a$ satisfies the nonlinear constant-mean-curvature equation \eqref{profileequation}. 

Although the dependence of the period $T_a$ on $a$ is not readily apparent in the axial parametrisation, parametrising the generating curve by arclength yields a considerably simpler description. As we demonstrate below, the squared radial coordinate then satisfies a second-order linear equation with constant coefficients, and the meridian length over one full Delaunay period is independent of $a$.

Let us be more precise. We define the signed arclength from the bulge
$s=0$ by
\begin{equation}\label{arclengthdefinition}
\sigma_a(s)\defeq \int_0^s\sqrt{1+\bigl(f_a'(\tau)\bigr)^2}\,\dd\tau,
\quad s\in\mathbb R,
\end{equation}
so that for a point of the generating curve with axial coordinate $s$, the quantity $\sigma=\sigma_a(s)$ accounts for its signed arclength coordinate measured from the bulge. As $\sigma_a:\mathbb R\to\mathbb R$ is a smooth increasing bijection, the axial coordinate may conversely be expressed as a function of
arclength $s=s_a(\sigma)$, with $s_a\defeq\sigma_a^{-1}$.

We then define
\begin{equation}\label{rhodefinition}
\rho_a(\sigma) \defeq f_a\bigl(s_a(\sigma)\bigr),\quad \sigma\in\mathbb R,
\end{equation}
so that $\sigma\mapsto \bigl(\rho_a(\sigma),s_a(\sigma)\bigr)$ is the same generating curve parametrised by signed arclength. In addition, differentiating the inverse relation $\sigma_a\bigl(s_a(\sigma)\bigr)=\sigma$
with respect to $\sigma$ and applying the chain rule to
\eqref{rhodefinition} provides the two useful identities
\begin{equation}\label{arclengthderivatives}
\frac{\dd s_a}{\dd\sigma}(\sigma)=\frac{1}
{\sqrt{1+\bigl(f_a'(s_a(\sigma))\bigr)^2}},\qquad
\frac{\dd\rho_a}{\dd\sigma}(\sigma)=\frac{f_a'(s_a(\sigma))}
{\sqrt{1+\bigl(f_a'(s_a(\sigma))\bigr)^2}}.
\end{equation}
Finally, evaluating \eqref{conservedquantity} at $s=s_a(\sigma)$ and using
\eqref{rhodefinition} and \eqref{arclengthderivatives} we obtain
\begin{equation}\label{saderivative}
\frac{\dd s_a}{\dd\sigma}(\sigma)=\frac{f_a(s_a(\sigma))^2+c}{f_a(s_a(\sigma))}=\frac{\rho_a(\sigma)^2+c}{\rho_a(\sigma)}.
\end{equation}
We establish the following result.
\begin{prop}\label{arclengthrepresentation}
Let $a\in(0,\frac12)$, $b=1-a$ and $c=ab$. For all $\sigma \in \R$, the radial profile in \eqref{rhodefinition} satisfies
\begin{equation*}
\frac{\dd^2}{\dd\sigma^2}\bigl(\rho_a(\sigma)^2\bigr)+4\rho_a(\sigma)^2= 2(a^2+b^2),
\end{equation*}
and hence
\begin{equation}\label{rhoexplicit}
\rho_a(\sigma)^2=\frac{a^2+b^2}{2} +\frac{b^2-a^2}{2}\cos(2\sigma).
\end{equation}
Consequently, $\sigma=0,\frac{\pi}{2},\pi$ correspond to a bulge, the adjacent neck, and the next bulge, respectively, with
\[s_a\Big(\frac{\pi}{2}\Big)=\frac{T_a}{2},\quad s_a(\pi)=T_a,\]
thus a full axial period $T_a$ corresponds to the fixed arclength interval $[0,\pi]$, independently of $a$.
\end{prop}

\begin{proof}
We first derive the equation satisfied by $\rho_a(\sigma)^2=f_a(s_a(\sigma))^2$. Differentiating 
\eqref{arclengthderivatives} with respect to $\sigma$ and using the chain rule we find
\[
\frac{\dd^2\rho_a}{\dd\sigma^2}(\sigma)=\frac{f_a''(s_a(\sigma))}
{\left(1+\bigl(f_a'(s_a(\sigma))\bigr)^2\right)^2},
\]
which further yields
\begin{equation}\label{rho2second}
\frac{\dd^2}{\dd\sigma^2}\bigl(\rho_a^2\bigr)=\frac{2\bigl(f_a'(s_a(\sigma))\bigr)^2}
{1+\bigl(f_a'(s_a(\sigma))\bigr)^2}+\frac{2f_a(s_a(\sigma))f_a''(s_a(\sigma))}
{\left(1+\bigl(f_a'(s_a(\sigma))\bigr)^2\right)^2}.
\end{equation}
We now evaluate the Delaunay equation \eqref{profileequation} at
$s=s_a(\sigma)$ and multiply by $
\frac{f_a(s_a(\sigma))}
{\sqrt{1+\bigl(f_a'(s_a(\sigma))\bigr)^2}}$ to get
\[\frac{f_a(s_a(\sigma))f_a''(s_a(\sigma))}
{\left(1+\bigl(f_a'(s_a(\sigma))\bigr)^2\right)^2}=
\frac{1}{1+\bigl(f_a'(s_a(\sigma))\bigr)^2}-
2\frac{f_a(s_a(\sigma))}
{\sqrt{1+\bigl(f_a'(s_a(\sigma))\bigr)^2}}.\]
Substituting the above into \eqref{rho2second} and then using \eqref{conservedquantity} once more, we arrive at the equation
\begin{equation}\label{oscequation}
\frac{\dd^2}{\dd\sigma^2}\bigl(\rho_a^2\bigr)+4\rho_a^2=2-4c=2(a^2+b^2),
\end{equation}
where we have used that $a+b=1$ and $c=ab.$ Since the initial conditions are $\rho_{a}(0)^2=b^2$ and $\frac{\dd}{\dd \sigma}(\rho_{a}^2)(0)=0,$ the solution of  \eqref{oscequation} reads
\[\rho_a(\sigma)^2=\frac{a^2+b^2}{2}+\frac{b^2-a^2}{2}\cos(2\sigma).\]
Finally, in order to identify the corresponding axial coordinates we
differentiate \eqref{rhoexplicit} to get
\begin{equation}\label{derivrhosquared}
\frac{\dd}{\dd\sigma}\bigl(\rho_a(\sigma)^2\bigr)=2\rho_a(\sigma)\frac{\dd\rho_a}{\dd\sigma}(\sigma)=-(b^2-a^2)\sin(2\sigma),
\end{equation}
hence using that $\rho_a>0$ and $b>a$, \eqref{derivrhosquared} and
\eqref{rhoexplicit} show that $\rho_a$ strictly decreases 
from $b$ to $a$ on $[0,\frac{\pi}{2}]$ and then strictly increases from $a$ to $b$ on $[\frac{\pi}{2},\pi]$. Recalling \eqref{rhodefinition} and using that $s_a$ is strictly increasing with $s_a(0)=0$, we deduce that
$\sigma=\frac{\pi}{2}$ and $\sigma=\pi$ correspond to the first neck and the next bulge after $s=0$, respectively, where we conclude that
\[
s_a\left(\frac{\pi}{2}\right)=\frac{T_a}{2},\quad s_a(\pi)=T_a.\]
\end{proof}
\medskip

Proposition \ref{arclengthrepresentation} and \eqref{saderivative} allow us to study the dependence of \(T_a\) on \(a\) over the fixed arclength interval \([0,\pi]\), leading to the following monotonicity results.

\begin{prop}\label{coefficientmonotonicity}
For every $a\in(0,\frac{1}{2})$, it holds
\[
\frac{\dd T_a}{\dd a}>0, \qquad \frac{\dd I_a}{\dd a}>0, \qquad \frac{\dd}{\dd a}\bigl(2T_aI_a\bigr)>0,
\]
and the maps
\[
T_a:\Big(0,\frac{1}{2}\Big)\rightarrow(2,\pi),
\qquad
I_a:\Big(0,\frac{1}{2}\Big)
\rightarrow\Big(0,\frac{\pi}{4}\Big),
\qquad
2T_aI_a:\Big(0,\frac{1}{2}\Big)
\rightarrow\left(0,\frac{\pi^2}{2}\right)
\]
are strictly increasing bijections.
\end{prop}

\begin{proof}
As a preliminary step, we define
\[\delta\defeq b-a\in(0,1),\]
so that in view of $a+b=1$ and $c=ab$, we have
\[
a=\frac{1-\delta}{2},\qquad
b=\frac{1+\delta}{2},\qquad
c=\frac{1-\delta^2}{4}.
\]
Invoking Proposition \ref{arclengthrepresentation}, the identity in
\eqref{saderivative}, and the properties
$s_a(0)=0$ and $s_a(\pi)=T_a$, we write
\begin{equation}\label{perioddelta}
T_{\frac{1-\delta}{2}}=\int_0^\pi F_\delta(\sigma)\,\dd\sigma,
\quad \mbox{with}\quad F_\delta(\sigma)\defeq
\frac{1+\delta\cos(2\sigma)}
{\sqrt{1+\delta^2+2\delta\cos(2\sigma)}}.
\end{equation}
To prove the strict monotonicity of $T_{\frac{1-\delta}{2}}$, we fix
$\delta_0\in(0,1)$ and choose $\eta>0$ sufficiently small so that
$[\delta_0-\eta,\delta_0+\eta]\subset(0,1)$. For every $\delta$
in this interval and $\sigma\in[0,\pi]$ it holds
\[
1+\delta^2+2\delta\cos(2\sigma)=(1-\delta)^2+4\delta\cos^2\sigma
\geq(1-\delta_0-\eta)^2>0,
\]
thus $F_\delta(\sigma)$ and its derivative with respect
to $\delta$ are uniformly bounded, while differentiation
under the integral sign in \eqref{perioddelta} and the chain rule give
\begin{equation}\label{periodderivative}
\frac{\dd}{\dd\delta}T_{\frac{1-\delta}{2}}=-\delta\int_0^\pi
\frac{\sin^2(2\sigma)}
{\left(1+\delta^2+2\delta\cos(2\sigma)\right)^{\frac32}}
\,\dd\sigma<0, \quad \frac{\dd T_a}{\dd a}=
-2\Big.
\frac{\dd}{\dd\delta}T_{\frac{1-\delta}{2}}
\Big|_{\delta=1-2a}>0.
\end{equation}
To determine the range of $T_a$, we first observe that the integrand in \eqref{perioddelta} satisfies $0<F_\delta(\sigma)\leq1$ for every $\delta\in(0,1)$ and
$\sigma\in[0,\pi]$, while one can additionally verify that
\[\begin{aligned}
&F_\delta(\sigma)\rightarrow1 \quad\text{as}\quad \delta\to0^+
\quad\text{for every }\sigma\in[0,\pi],\\
&F_\delta(\sigma)\rightarrow|\cos\sigma|
\quad\text{as}\quad \delta\to1^- \quad\text{for a.e. }\sigma\in[0,\pi].
\end{aligned}\]
As a result, dominated convergence in \eqref{perioddelta} together with $\delta=1-2a$ gives
\begin{equation}\label{endpointsTa}
\lim_{a\to\frac12^-}T_a=\pi,
\qquad
\lim_{a\to0^+}T_a=2,
\end{equation}
where the strict monotonicity also provides
the desired bijection for $T_a$. Finally, using \eqref{exactidentity} and \eqref{periodderivative} we directly obtain 
\[\begin{aligned}
\frac{\dd I_a}{\dd a}&=(1-2a)T_a+a(1-a)\frac{\dd T_a}{\dd a}>0,\\
\frac{\dd}{\dd a}\bigl(2T_aI_a\bigr)&=
2(1-2a)T_a^2+4a(1-a)T_a\frac{\dd T_a}{\dd a}>0. 
\end{aligned}\] 
Likewise, the endpoint limits for $T_a$ in \eqref{endpointsTa} yield
\begin{equation} \label{limitsinproof}
\lim_{a\to0^+}I_a=0,
\qquad
\lim_{a\to\frac12^-}I_a=\frac{\pi}{4}, \qquad \lim_{a\to0^+}2T_aI_a=0,
\qquad
\lim_{a\to\frac12^-}2T_aI_a=\frac{\pi^2}{2},
\end{equation}
while the remaining two bijections for $I_a$ and $2T_aI_a$ follow from the established strict monotonicity.
\end{proof}
\medskip

Taking into account the preceding results, we next derive an exact relation between the volume $V_a$ and surface area $A_a$ of one Delaunay period and determine their monotonicity properties and endpoint limits.
\begin{prop}\label{volumemonotonicity}
For every \(a\in(0,\frac12)\), the volume and the surface area of one period in \eqref{periodgeometry} satisfy the exact relation
\begin{equation}\label{areaidentity}
A_a=3V_a+\pi I_a.
\end{equation}
Moreover, both \(V_a\) and \(A_a\) are strictly decreasing functions of \(a\), and their endpoint limits are
\begin{equation}\label{volumearealimits}
\lim_{a\to0^+}V_a=\frac{4\pi}{3},\qquad
\lim_{a\to\frac12^-}V_a=\frac{\pi^2}{4},\qquad
\lim_{a\to0^+}A_a=4\pi,\qquad
\lim_{a\to\frac12^-}A_a=\pi^2.
\end{equation}
In particular, \(a\mapsto V_a\) is a bijection onto $\left(\frac{\pi^2}{4},\frac{4\pi}{3}\right),$ while 
$a\mapsto\frac{2T_aI_a}{V_a}$ and $a\mapsto\frac{A_a}{V_a}$ are strictly increasing, with endpoint limits
\begin{equation}\label{quotientlimits}
\lim_{a\to0^+}\frac{2T_aI_a}{V_a}=0,\qquad
\lim_{a\to\frac12^-}\frac{2T_aI_a}{V_a}=2,\qquad
\lim_{a\to0^+}\frac{A_a}{V_a}=3,\qquad
\lim_{a\to\frac12^-}\frac{A_a}{V_a}=4.
\end{equation}

\end{prop}

\begin{proof}
As a first step, we recall that $b=1-a,\,c=ab$ and use \eqref{profileequation} and \eqref{conservedquantity} to obtain the three identities
\begin{equation}\label{threeuseful}
\sqrt{1+(f_a')^2}=\frac{f_a}{f_a^2+c},
\qquad (f_a')^2=\frac{f_a^2}{(f_a^2+c)^2}-1, \qquad f_a''=\frac{f_a(c-f_a^2)}{(f_a^2+c)^3}.
\end{equation}
Next, using \eqref{threeuseful} we compute
\begin{equation}\label{areaexactderivative}
\frac{\dd}{\dd s}\left(f_a(f_a^2+c)f_a'\right)=2f_a\sqrt{1+(f_a')^2}-3f_a^2-c,
\end{equation}
so that integrating \eqref{areaexactderivative} over
$[-\frac{T_a}{2},\frac{T_a}{2}]$ and multiplying by $\pi$ gives \eqref{areaidentity}, where we have used that $f_a'\left(-\frac{T_a}{2}\right)=f_a'\left(\frac{T_a}{2}\right)=0$ together with \eqref{periodgeometry} and \eqref{exactidentity}.

To proceed, setting again $\delta= b-a=1-2a\in(0,1)$, we use the arclength coordinate of \eqref{arclengthdefinition} in combination with \eqref{rhodefinition}, \eqref{arclengthderivatives} and Proposition \ref{arclengthrepresentation} to write  
\begin{equation} \label{transformedvolumeint}
V_a=\pi\int_0^{T_a}f_a(s)^2\,\dd s=\pi\int_0^\pi\rho_a(\sigma)^2
\frac{\dd s_a}{\dd\sigma}(\sigma)\,\dd\sigma,
\end{equation}
which takes the equivalent form
\begin{equation}\label{volumedelta}
V_{\frac{1-\delta}{2}}=\frac{\pi}{4}\int_0^\pi
\bigl(1+\delta\cos(2\sigma)\bigr)
\sqrt{1+\delta^2+2\delta\cos(2\sigma)}
\,\dd\sigma.
\end{equation}
For $\delta$ in any compact subinterval of $(0,1)$, the integrand
in \eqref{volumedelta} and its derivative with respect to $\delta$
are uniformly bounded for $\sigma \in[0,\pi]$, hence differentiation under the
integral sign gives
\begin{equation} \label{derivdeltavolume}
\begin{aligned}
\frac{4}{\pi}
\frac{\dd}{\dd\delta}V_{\frac{1-\delta}{2}}&=\int_0^\pi
\bigg[\cos(2\sigma)\sqrt{1+\delta^2+2\delta\cos(2\sigma)}
+\frac{\bigl(1+\delta\cos(2\sigma)\bigr)
\bigl(\delta+\cos(2\sigma)\bigr)}
{\sqrt{1+\delta^2+2\delta\cos(2\sigma)}}\bigg]\dd\sigma\\
&=\delta\int_0^\pi\frac{(1+\delta^2)\bigl(1+\cos^2(2\sigma)\bigr)+4\delta\cos(2\sigma)}{\bigl(1+\delta^2+2\delta\cos(2\sigma)\bigr)^{\frac32}}\,\dd\sigma>0,
\end{aligned}
\end{equation}
where we integrated by parts and used that 
\[(1+\delta^2)\bigl(1+\cos^2(2\sigma)\bigr)
+4\delta\cos(2\sigma)=(1+\delta^2)\Big(\cos(2\sigma)+\frac{2\delta}{1+\delta^2}
\Big)^2+\frac{(1-\delta^2)^2}{1+\delta^2}>0.\]
As a consequence, one finds 
\begin{equation} \label{monotonicityvolumeinproof}
\frac{\dd V_a}{\dd a}=-2\Big.\frac{\dd}{\dd\delta}V_{\frac{1-\delta}{2}}
\Big|_{\delta=1-2a}<0.
\end{equation}
To prove \eqref{volumearealimits}, for notational convenience we first denote the integrand in \eqref{volumedelta} by
\[
G_\delta(\sigma)\defeq\bigl(1+\delta\cos(2\sigma)\bigr)\sqrt{1+\delta^2+2\delta\cos(2\sigma)},
\]
and observe that for every \(\delta\in(0,1)\) and \(\sigma\in[0,\pi]\) it holds 
\[
0<G_\delta(\sigma)\leq4,\qquad
\lim_{\delta\to0^+}G_\delta(\sigma)=1,\qquad
\lim_{\delta\to1^-}G_\delta(\sigma)=4|\cos\sigma|^3.
\]
As such, dominated convergence in \eqref{volumedelta} establishes the limits
\begin{equation}\label{volumelimitsinproof}
\lim_{a\to\frac{1}{2}^-}V_a=\frac{\pi^2}{4},
\qquad
\lim_{a\to0^+}V_a=\frac{4\pi}{3},
\end{equation}
while the continuity and strict monotonicity of $V_a$ also
show that $V_a:(0,\frac{1}{2})\to
(\frac{\pi^2}{4},\frac{4\pi}{3})$ is a bijection. In addition, \eqref{areaidentity} gives $\frac{A_a}{V_a}=3+\pi\frac{I_a}{V_a}$, which is strictly increasing since $I_a$ is positive and strictly increasing, whereas $V_a$ is positive and strictly
decreasing.

Employing the same change of variables $s=s_a(\sigma)$ as in
\eqref{transformedvolumeint} to the area formula in
\eqref{periodgeometry}, by \eqref{arclengthderivatives} and
Proposition \ref{arclengthrepresentation} we obtain
\[
A_a=2\pi\int_0^\pi\rho_a(\sigma)\,\dd\sigma,
\]
while using \eqref{rhoexplicit} and $\delta=1-2a$ we further write
\[
A_{\frac{1-\delta}{2}}=\pi\int_0^\pi\sqrt{1+\delta^2+2\delta\cos(2\sigma)}
\,\dd\sigma.
\]
Similarly to \eqref{derivdeltavolume}, differentiation under the
integral sign followed by integration by parts yields
\[
\frac{\dd}{\dd\delta}A_{\frac{1-\delta}{2}}=\pi\int_0^\pi
\frac{\delta+\cos(2\sigma)}
{\sqrt{1+\delta^2+2\delta\cos(2\sigma)}}
\,\dd\sigma=\pi\delta\int_0^\pi
\frac{\bigl(\delta+\cos(2\sigma)\bigr)^2}
{\bigl(1+\delta^2+2\delta\cos(2\sigma)\bigr)^{\frac{3}{2}}}
\,\dd\sigma>0,
\]
while combining \eqref{areaidentity} and \eqref{volumelimitsinproof} with the
endpoint limits for $I_a$ in Proposition
\ref{coefficientmonotonicity} we infer that
\begin{equation} \label{limitAainproof}
\lim_{a\to0^+}A_a=4\pi,\qquad \lim_{a\to\frac12^-}A_a=\pi^2.
\end{equation}
Finally, using \eqref{monotonicityvolumeinproof} and Proposition  \ref{coefficientmonotonicity} once more we deduce that 
\[
\begin{aligned}
\frac{\dd}{\dd a}
\Big(\frac{2T_aI_a}{V_a}\Big)&=\frac{\frac{\dd}{\dd a}\bigl(2T_aI_a\bigr)V_a
-2T_aI_a \frac{\dd V_a}{\dd a}}{V_a^2}>0, \\
\frac{\dd}{\dd a}\Big(\frac{I_a}{V_a}\Big)&=\frac{\frac{\dd I_a}{\dd a}V_a-I_a\frac{\dd V_a}{\dd a}}{V_a^2}>0,
\end{aligned}
\]
while the limits in \eqref{quotientlimits} follow from a combination of \eqref{limitsinproof}, \eqref{volumelimitsinproof} and \eqref{limitAainproof}.  
\end{proof}
\section{Consequences for Delaunay-like liquid drop equilibria}
\label{consequences}
We now return to the construction in \cite{dPMZ} and examine its dependence on the reference Delaunay unduloid. As we show below, the positivity of $I_a$ established by Theorem~\ref{maintheorem} permits the construction for every $a\in(0,\frac{1}{2})$, while the monotonicity results of Section~\ref{monotonicity} extend the prescribed-mass argument to a neighbourhood of any interior reference neck. 

Before stating these results, we record two corrections to \cite{dPMZ}*{(2.5)-(2.6)} concerning the leading coupling coefficient and the small-neck asymptotics of $I_a$. For clarity, we note that all the asymptotic estimates below involving $n$ are taken as $n\to\infty$, uniformly for
$a$ in a fixed nonempty compact interval $K\subset(0,\frac12)$.

In the projection following \cite{dPMZ}*{(6.5)}, the logarithmic term contains the integral
\[
\int_0^{2\pi}\int_{-\frac{T_a}{2}}^{\frac{T_a}{2}}
f_a(s)^2\sin^2\theta\,\dd s\,\dd\theta=V_a.
\]
Retaining this factor and using that $|\Omega_0|=V_a$ and $2\pi R=nT_a$, we arrive at the scalar solvability condition
\begin{equation}\label{correctedsolvability}
\frac{4\pi^2I_a}{nT_a}-\frac{2\pi V_a^2}{nT_a^2}\gamma\log n+\mathcal{O}\left(\frac{\gamma}{n}\right)=0,
\end{equation}
where $\Omega_0$ denotes the region enclosed by the reference unduloid within one axial period. Since $\frac{2\pi T_aI_a}{V_a^2}$ is positive and continuous in $K$, the fixed-point argument in \cite{dPMZ}*{Section 6} applies with the corrected leading coefficient in \eqref{correctedsolvability}, yielding
\begin{equation}\label{correctedcoupling}
\gamma_{n,a}=\frac{2\pi T_aI_a}{V_a^2}\frac{1}{\log n}\bigl(1+o(1)\bigr)
\quad\text{as} \quad n\to\infty.
\end{equation}
Regarding the small-neck asymptotics of $I_a$, Proposition \ref{arclengthrepresentation} together with \eqref{saderivative} and \eqref{exactidentity} gives
\[
\rho_a(\sigma)=\sqrt{b^2\cos^2\sigma+a^2\sin^2\sigma}, \qquad I_a=2ab\int_0^{\frac{\pi}{2}}\left(\rho_a+\frac{ab}{\rho_a}\right)\,\dd\sigma,
\]
where we recall that $b=1-a$. Due to the bounds $b\cos\sigma\leq\rho_a(\sigma)\leq b\cos\sigma+a\sin\sigma$ in $[0,\frac{\pi}{2}]$ we readily have $\int_0^{\frac{\pi}{2}}\rho_a\,\dd\sigma=b+\mathcal{O}(a)$, while for the remaining integral we use the change of variables $t=\cot\sigma$ to get
\[
\mathfrak J_a\defeq\int_0^{\frac{\pi}{2}}\frac{\dd\sigma}{\rho_a(\sigma)}=\int_0^\infty\frac{\dd t}{\sqrt{1+t^2}\sqrt{a^2+b^2t^2}}.
\]
Since $0\leq1-\frac{1}{\sqrt{1+t^2}}\leq\frac{t^2}{2}$ for $t\in[0,1]$ and $b\to 1$ as $a \to 0^{+}$, one finds
\[
\left|\mathfrak J_a-\int_0^1\frac{\dd t}{\sqrt{a^2+b^2t^2}}\right|
\leq\frac{1}{2b}\int_0^1t\,\dd t+\frac{1}{b}\int_1^\infty\frac{\dd t}{t^2}=\frac{5}{4b},
\]
hence we infer that
\[
\mathfrak J_a=\frac{1}{b}\log\left(\frac{b+\sqrt{a^2+b^2}}{a}\right)+\mathcal{O}(1)=\frac{|\log a|}{b}+\mathcal{O}(1).
\]
Therefore, substituting the above into the expression for $I_a$ we obtain
\[
I_a=2ab^2+2a^2b|\log a|+\mathcal{O}(a^2)=2a+2a^2|\log a|+\mathcal{O}(a^2)\quad\text{as} \quad a\to0^+.
\]
Returning to the corrected asymptotic formula \eqref{correctedcoupling}, we now
compute the physical mass. Under the scaling
$\Omega_{n,a}=(\gamma_{n,a})^{\frac 13}\widetilde\Omega_{n,a}$,
we have $m_n(a)=\gamma_{n,a}|\widetilde\Omega_{n,a}|$, hence the volume estimate \eqref{dPMZestimates} gives
\[
m_n(a)=\frac{n}{\log n}\frac{2\pi T_a I_a}{V_a}(1+o(1)) \quad\text{as}\quad n\to\infty.
\]
This motivates the definition of the
leading mass coefficient
\begin{equation}\label{Mcaligraphic}
\mathcal M(a)\defeq
\frac{2\pi T_a I_a}{V_a}= \frac{2\pi a(1-a)T_a^2}{V_a}, \quad a\in\Big(0,\frac12\Big),
\end{equation}
which satisfies $\mathcal M(a)>0$ and $\mathcal M'(a)>0$
due to Theorem \ref{maintheorem} and Proposition \ref{volumemonotonicity}.
\begin{corollary}\label{uniformconstruction}
Let $K\subset(0,\frac{1}{2})$ be a nonempty compact interval.
Then, there exists an integer $n_K\geq3$ such that for every integer
$n\geq n_K$ and every $a\in K$, the construction in
\cite{dPMZ} produces a bounded domain $\Omega_{n,a}$ with
smooth embedded boundary satisfying \eqref{liquidequation}.
The associated coupling parameters $\gamma_{n,a}>0$ and masses
$m_n(a)\defeq|\Omega_{n,a}|$ depend continuously on $a$ and
satisfy
\begin{align}
\gamma_{n,a}&=\frac{\mathcal M(a)}{V_a\log n}\bigl(1+r_n(a)\bigr),
\label{gammafullrange}\\
m_n(a)&=\frac{n}{\log n}\mathcal M(a)\bigl(1+\varepsilon_n(a)\bigr),
\label{massreduction}
\end{align}
where $r_n,\varepsilon_n\in C(K)$ and
\begin{equation}\label{uniformerror}
\sup_{a\in K}\bigl(|r_n(a)|+|\varepsilon_n(a)|\bigr)
\rightarrow0\qquad\text{as} \quad n\to\infty.
\end{equation}
Moreover, for every $\Lambda\in(0,\frac{8}{\pi})$ there exists
a unique $a_\Lambda\in(0,\frac{1}{2})$ such that
\begin{equation}\label{couplingvalue}
\frac{\mathcal M(a_\Lambda)}{V_{a_\Lambda}}=\Lambda,
\qquad
\gamma_{n,a_\Lambda}\log n\rightarrow\Lambda \quad\text{as}\quad n\to\infty.
\end{equation}
\end{corollary}
\begin{proof}
Using \eqref{fluxidentity} and the compactness of $K$ we have
$\inf\limits_{a\in K}I_a>0$, thus the construction recalled above applies uniformly for $a\in K$, yielding a common threshold $n_K\geq 3$. The contraction arguments in \cite{dPMZ}*{Section 6} give continuous dependence of the coupling parameter and the boundary perturbation on $a$, and hence also of the physical mass. The uniform asymptotics in \eqref{correctedcoupling} and \eqref{dPMZestimates}, together with \eqref{Mcaligraphic} and the continuous dependence on $a$ yield \eqref{gammafullrange}-\eqref{uniformerror}, with $r_n,\varepsilon_n\in C(K)$. Finally, Theorem \ref{maintheorem} and Proposition \ref{volumemonotonicity} show that $a\mapsto\frac{\mathcal M(a)}{V_a}$ is continuous and strictly increasing, with endpoint limits $0$ and $\frac{8}{\pi}$. Accordingly, each $\Lambda\in(0,\frac{8}{\pi})$ determines a unique $a_\Lambda\in(0,\frac 12)$, therefore applying
\eqref{gammafullrange} and \eqref{uniformerror} on a compact
interval $K_\Lambda\subset(0,\frac12)$ containing $a_\Lambda$
gives the limit in \eqref{couplingvalue}.
\end{proof}
\medskip

For any fixed $a\in(0,\frac 12)$, Corollary \ref{uniformconstruction} provides a sequence of masses $m_n(a)$ indexed by $n$. The uniform asymptotics \eqref{massreduction} combined with the continuity of $m_n$ and the strict monotonicity of  $\mathcal M$ in \eqref{Mcaligraphic} allow every sufficiently large mass to be realised by choosing $n$ and then adjusting $a$ near any prescribed interior reference neck.

\begin{corollary}
Fix $a_0\in(0,\frac{1}{2})$. Then, every sufficiently large mass $m$ is
realised by a bounded domain $\Omega_{n_m,a_m}\subset\mathbb R^3$
with smooth embedded boundary satisfying \eqref{liquidequation},
where $n_m\in\mathbb N$, $a_m\in(0,\frac{1}{2})$, and
\[
|\Omega_{n_m,a_m}|=m, \qquad n_m\to\infty, \qquad a_m\to a_0 \quad\text{as}\quad m\to\infty.
\]
The integers $n_m$ may be chosen so that as $m\to\infty$,
\begin{align}
\frac{n_m}{\log n_m}&=\frac{m}{\mathcal M(a_0)}\bigl(1+o(1)\bigr), \label{nasymptotic}\\
n_m&=\frac{m}{\mathcal M(a_0)}
\left(\log m+\log\log m-\log\mathcal M(a_0)+o(1)\right), \label{ninversion}
\end{align}
where $\mathcal{M}(a)$ is defined in \eqref{Mcaligraphic}.
\end{corollary}

\begin{proof}
We take $K=[a_-,a_+]\subset(0,\frac{1}{2})$ with
$a_-<a_0<a_+$, and consider the domains $\Omega_{n,a}$
provided by Corollary \ref{uniformconstruction} for
$n\geq n_K$ and $a\in K$. Following the prescribed-mass
argument in \cite{dPMZ}*{Proof of Theorem 1}, we first choose $n_m$ and then adjust the neck parameter.

More specifically, we set $\zeta(x)\defeq\frac{x}{\log x}$ for $x>e$, define $\mathcal M_0\defeq\mathcal M(a_0)>0$, and select
\begin{equation}\label{integerpart}
n_m=\Big\lfloor
\frac{m}{\mathcal M_0}
\left(\log m+\log\log m-\log\mathcal M_0\right)
\Big\rfloor,
\end{equation}
where $\lfloor \cdot \rfloor$ denotes the integer part function. Due to the estimates
\[
n_m=\frac{m\log m}{\mathcal M_0}(1+o(1)), \qquad
\log n_m=\log m+\log\log m-\log\mathcal M_0+o(1),
\]
one can verify that
\begin{equation}\label{targetlimit}
n_m\to\infty, \qquad \frac{m}{\zeta(n_m)}
=\frac{m\log n_m}{n_m}
\rightarrow\mathcal M_0 \quad \mbox{as}\quad m\to \infty,
\end{equation}
and therefore $n_m\geq n_K$ for all sufficiently large $m$.

The uniform mass asymptotics \eqref{massreduction} and \eqref{uniformerror} then yield 
\[
\frac{m_{n_m}(a_\pm)}{\zeta(n_m)}=\mathcal M(a_\pm)\bigl(1+\varepsilon_{n_m}(a_\pm)\bigr)\rightarrow\mathcal M(a_\pm) \quad\text{as}\quad m\to\infty,
\]
where \eqref{targetlimit} and the strict inequalities $\mathcal M(a_-)<\mathcal M_0<\mathcal M(a_+)$ further imply
\[
m_{n_m}(a_-)<m<m_{n_m}(a_+)
\]
for all sufficiently large $m$.

Next, the intermediate value theorem provides $a_m\in(a_-,a_+)$ satisfying $m_{n_m}(a_m)=m$, so that \eqref{massreduction}, \eqref{uniformerror} and \eqref{targetlimit} give
\[
\mathcal M(a_m)=\frac{m}{\zeta(n_m)\bigl(1+\varepsilon_{n_m}(a_m)\bigr)}\rightarrow\mathcal M_0=\mathcal M(a_0) \quad\text{as} \quad m\to\infty.
\]
Consequently, $a_m\to a_0$ as $m\to\infty$ by the continuity of $\mathcal M^{-1}$, while \eqref{nasymptotic} and \eqref{ninversion} follow from \eqref{targetlimit} and \eqref{integerpart}, respectively.
\end{proof}

\appendix
\section{The flux identity via elliptic integrals}
In this section, we provide an alternative evaluation of the coefficient
$I_a$ in \eqref{coefficient} through elliptic integral recurrences, independently of Lemma \ref{pointwiseidentity}.

For fixed $a\in(0,\frac{1}{2})$ we recall that $b=1-a,\,c=ab$ and set $d\defeq a^2+b^2=1-2c$. Moreover, we define
\[
k\defeq\sqrt{1-\frac{a^2}{b^2}}\in(0,1),
\]
and consider the complete elliptic integrals 
\begin{equation}\label{ellipticdefinitions}
K(k)\defeq\int_0^{\frac{\pi}{2}}\frac{\dd\theta}{\sqrt{1-k^2\sin^2\theta}},
\qquad E(k)\defeq\int_0^{\frac{\pi}{2}}\sqrt{1-k^2\sin^2\theta}\,\dd\theta.
\end{equation}
We also employ the auxiliary elliptic integrals
\begin{equation}\label{ellipticmoments}
J_j\defeq\int_a^b\frac{x^{2j}}{\sqrt{(x^2-a^2)(b^2-x^2)}}\,\dd x,
\qquad j\in\{-1,0,1,2,3\},
\end{equation}
and note that the change of variables
$x=b\sqrt{1-k^2\sin^2\theta}\,$ for $\theta\in(0,\frac{\pi}{2})$ gives
\[
\frac{\dd x}{\sqrt{(x^2-a^2)(b^2-x^2)}}=-\frac{\dd\theta}{b\sqrt{1-k^2\sin^2\theta}}.
\]
Using the above identity, we combine \eqref{ellipticdefinitions} and
\eqref{ellipticmoments} to find
\begin{equation}\label{ellipticbasicmoments}
J_0=\frac{K(k)}{b}, \qquad J_1=bE(k).
\end{equation}
On the decreasing branch $s\in(0,\frac{T_a}{2})$, the substitution $x=f_a(s)$ together with \eqref{decreasingbranch} leads to
\begin{equation}\label{periodradiusintegral}
\dd s=-\frac{x^2+c}{\sqrt{(x^2-a^2)(b^2-x^2)}}\,\dd x, \quad \frac{T_a}{2}=\int_a^b\frac{x^2+c}{\sqrt{(x^2-a^2)(b^2-x^2)}}\,\dd x = J_1+cJ_0,
\end{equation}
hence \eqref{ellipticbasicmoments} and \eqref{periodradiusintegral} allow to express the period in terms of the elliptic integrals in \eqref{ellipticdefinitions} as
\begin{equation}\label{ellipticperiod}
T_a=2\bigl(bE(k)+aK(k)\bigr).
\end{equation}
In view of \eqref{threeuseful}, we then rewrite the integrand in  \eqref{coefficient} as
\[\Phi_{a}(s)=\frac{c^3}{f_a(s)^2}+c(2-3c)+(4-9c)f_a(s)^2-5f_a(s)^4,
\]
so that the same change of variables $x=f_a(s)$ yields
\begin{equation}\label{ellipticcoefficient}
I_a=c^4J_{-1}+2c^2(1-c)J_0+6cdJ_1+(4-14c)J_2-5J_3.
\end{equation}
Next, in order to express $I_a$ in terms of $J_0$ and $J_1$ in \eqref{ellipticbasicmoments}, for $j=0,1,2$ and $x\in(a,b)$ we use the identity
\[(x^2-a^2)(b^2-x^2)=-x^4+dx^2-c^2\]
to compute
\begin{equation}\label{Recurrencediff}
\frac{\dd}{\dd x}\left(x^{2j-1}\sqrt{(x^2-a^2)(b^2-x^2)}\right)
=\frac{-(2j+1)x^{2j+2}+2jd\,x^{2j}-(2j-1)c^2x^{2j-2}
}{\sqrt{(x^2-a^2)(b^2-x^2)}}.
\end{equation}
Since the derivative in \eqref{Recurrencediff} is integrable in $(a,b)$ and $x^{2j-1}\sqrt{(x^2-a^2)(b^2-x^2)}$ vanishes at both endpoints, integrating provides the recurrence relations
\begin{equation}\label{ellipticrecurrence}
J_1=c^2J_{-1},\qquad 3J_2=2dJ_1-c^2J_0, \qquad 5J_3=4dJ_2-3c^2J_1.
\end{equation}
Using the first and third identities in \eqref{ellipticrecurrence} to eliminate $J_{-1}$ and $J_3$ from \eqref{ellipticcoefficient} and then the second
relation to simplify $J_2$, we arrive at
\[
\begin{aligned}
I_a=2c^2J_0+2cJ_1=2c\bigl(bE(k)+aK(k)\bigr) =cT_a,
\end{aligned}
\]
where we have used that $d+2c=1$, \eqref{ellipticbasicmoments} and
\eqref{ellipticperiod}.
\section*{Acknowledgements}
The research of M. del Pino is supported by the Royal Society
Research Professorship grant RP-R1-180114 and by the ERC/UKRI
Horizon Europe grant ASYMEVOL, EP/Z000394/1. The authors acknowledge the use of ChatGPT (OpenAI) to assist with selected calculations. They have independently verified all mathematical arguments and calculations and take full responsibility for the content of the paper.
\bibliographystyle{amsplain}
\bibliography{references}\end{document}